\documentclass[a4paper,12pt]{amsart}

\providecommand{\bysame}{\leavevmode\hbox to3em{\hrulefill}\thinspace}
\providecommand{\MR}{\relax\ifhmode\unskip\space\fi MR }

\providecommand{\href}[2]{#2}

\usepackage{enumerate, booktabs}
\usepackage{amsmath, amsfonts, amssymb, amsthm}
\usepackage[all]{xy}
\usepackage[mathcal]{euscript}
\usepackage{mathpazo}

\numberwithin{equation}{section}

\theoremstyle{plain}
\newtheorem{thm}{Theorem}[section]
\newtheorem{prop}[thm]{Proposition}
\newtheorem{lem}[thm]{Lemma}

\newtheorem*{thm*}{Theorem}
\newtheorem*{main}{Main Theorem}

\theoremstyle{definition}
\newtheorem{defn}[thm]{Definition}

\theoremstyle{remark}
\newtheorem*{rem}{Remark}

\newcommand{\mbb}[1]{\mathbb{#1}}
\newcommand{\ol}[1]{\overline{#1}}
\newcommand{\lie}[1]{{\mathfrak{#1}}}
\newcommand{\abs}[1]{\lvert #1\rvert}
\newcommand{\norm}[1]{\lVert #1\rVert}

\title{Pseudoconvex non-Stein domains in primary Hopf surfaces}

\author{Christian Miebach}
\address{Laboratoire de Math\'ematiques Pures et Appliqu\'ees, CNRS-FR~2956,
Universit\'e du Littoral C\^ote d'Opale, 50, rue F.~Buisson, 62228 Calais Cedex,
France}
\email{miebach@lmpa.univ-littoral.fr}


\begin{document}

\begin{abstract}
We describe pseudoconvex non-Stein domains in primary Hopf surfaces using
techniques developed by Hirschowitz.
\end{abstract}

\maketitle

\section{Introduction}

Let $H$ be a primary Hopf surface. In~\cite{LY} Levenberg and Yamaguchi
characterize locally pseudoconvex domains $D\subset H$ ha\-ving smooth
real-analytic boundary that are not Stein. In this note we generalize their
result to arbitrary pseudoconvex domains using ideas developed by Hirschowitz
in~\cite{Hir1} and~\cite{Hir}. For the rea\-ders' convenience these ideas are
reviewed in a slightly generalized form in Section~2. In Section~3 we review
the structure of primary Hopf surfaces in order to describe a certain (singular)
holomorphic foliation $\mathcal{F}$ of $H$. This allows us to formulate the
following Main Theorem, which is proven in Sections~4 and~5.

\begin{main}
Let $D\subset H$ be a pseudoconvex domain. If $D$ is not Stein, then $D$ contains
with every point $p\in D$ the topological closure $\ol{F}_p$ of the leaf
$F\in\mathcal{F}$ passing through $p$.
\end{main}

I would like to thank Karl Oeljeklaus for helpful discussions on the subject of
this paper and Stefan Nemirovski for a suggestion on how to prove
Lemma~\ref{Lem:Stein}. I am also grateful to Peter Heinzner and the SFB/TR~12 for
an invitation to the Ruhr-Universit\"at Bochum where a part of this paper has
been written.

\section{A review of Hirschowitz' methods}

In this section we present the methods developed by Hirschowitz in~\cite{Hir}
in a slightly more general setup.

Let $X$ be a complex manifold with holomorphic tangent bundle $TX\to X$, and let
$\pi\colon\mbb{P}TX\to X$ be the projectivized holomorphic tangent bundle. A
continuous function on $X$ is called \emph{strictly} pluri\-subharmonic on $X$
if it is everywhere locally the sum of a continuous plurisubharmonic and a
smooth strictly plurisubharmonic function.

\begin{defn}
Let $\varphi\in\mathcal{C}(X)$ be plurisubharmonic. Then we define $S(\varphi)$
to be the set of $[v]\in\mbb{P}TX$ such that $\varphi$ is in a neighborhood of
$\pi[v]$ the sum of a plurisubharmonic function and a smooth function that is
strictly plurisubharmonic on any germ of a holomorphic curve defining $[v]$.
\end{defn}

\begin{lem}
Let $\varphi\in\mathcal{C}(X)$ be plurisubharmonic.
\begin{enumerate}[(1)]
\item The set $S(\varphi)$ is open in $\mbb{P}TX$.
\item If $S(\varphi)=\mbb{P}TX$, then $\varphi$ is strictly plurisubharmonic on
$X$.
\item If $\sum_k\varphi_k$ converges uniformly on compact subsets of $X$ where
$\varphi_k\in\mathcal{C}(X)$, then we have $S\bigl(\sum_k\varphi_k\bigr)
\supset\bigcup_k S(\varphi_k)$.
\end{enumerate}
\end{lem}

\begin{proof}
This is~\cite[Proposition~1.3]{Hir}.
\end{proof}

For any plurisubharmonic function $\varphi\in\mathcal{C}(X)$ we define
$C(\varphi)$ to be
\begin{equation*}
\mbb{P}TX\setminus
\bigl\{[v]\in\mbb{P}TX;\text{ $\varphi$ is smooth around
$\pi[v]$ and } \partial\varphi(v)\not=0\bigr\}
\end{equation*}
and then set
\begin{equation}\label{Eqn:C(X)}
C(X):=\bigcap_{\substack{\varphi\in\mathcal{C}(X)\\\text{plurisubharmonic}}}
C(\varphi).
\end{equation}
Every set $C(\varphi)$ (and thus $C(X)$) is closed in $\mbb{P}TX$. The next lemma
is a slight gene\-ralization of~\cite[Proposition~1.5]{Hir}.

\begin{lem}\label{Lem:StrPsh}
Let $X$ be a complex manifold and let $\Omega:=X\setminus\pi\bigl(C(X)\bigr)$.
Then there exists a plurisubharmonic function $\psi\in\mathcal{C}(X)$ which is
strictly plurisubharmonic on $\Omega$.
\end{lem}

\begin{proof}
Since $\pi\colon\mbb{P}TX\to X$ is proper, the set $\Omega$ is open in $X$. If
$\Omega$ is empty, there is nothing to prove. Therefore let us suppose that
$\Omega$ is a non-empty open subset of $X$. Consequently, $\mbb{P}TX\setminus
C(X)$ is non-empty.

For every $[v]\in\mbb{P}TX\setminus C(X)$ we find a plurisubharmonic function
$\varphi_{[v]}\in\mathcal{C}^\infty(X)$ with $\partial\varphi_{[v]}v\not=0$. We
claim that the function $\psi_{[v]}:=\exp\circ\,\varphi_{[v]}$ is strictly
plurisubharmonic in the direction of $[v]$. To see this, we calculate
\begin{equation*}
\partial\ol{\partial}\psi_{[v]}=e^{\varphi_{[v]}}\bigl( \partial\varphi_{[v]}
\wedge\ol{\partial}\varphi_{[v]} + \partial\ol{\partial}\varphi_{[v]}\bigr).
\end{equation*}
In other words, we obtain $[v]\in S(\psi_{[v]})$. Since $X$ has countable
topo\-logy, we get an open covering
\begin{equation*}
\mbb{P}TX\setminus C(X)\subseteq\bigcup_{k=1}^\infty S(\psi_k).
\end{equation*}
It is possible to find $\lambda_k>0$ such that $\sum_{k=1}^\infty\lambda_k
\psi_k$ converges uniformly on compact subsets of $X$. To prove this, choose a
countable exhaustion $X=\bigcup_jK_j$ by compact sets with $K_j\subset
\mathring{K}_{j+1}$. For every $j$ there are $\lambda_{k,j}>0$ such that
\begin{equation*}
\sum_{k=1}^\infty\lambda_{k,j}\norm{\psi_k}_{K_j}
\end{equation*}
converges. Since $\norm{\psi_k}_{K_j}\leq\norm{\psi_k}_{K_{j+1}}$ for every $j$,
we may suppose that $\lambda_{k,j'}\leq\lambda_{k,j}$ for all $j\leq j'$.
Defining $\lambda_k:=\lambda_{k,k}$ and noting that every compact subset
$K\subset X$ is contained in $K_{j_0}$ for some $j_0$, we conclude
\begin{align*}
\sum_{k=1}^\infty\lambda_k\norm{\psi_k}_K&\leq
\sum_{k=1}^\infty\lambda_k\norm{\psi_k}_{K_{j_0}}\\
&\leq\sum_{k=1}^{j_0}\lambda_k\norm{\psi_k}_{K_{j_0}}
+\sum_{k=j_0+1}^\infty\lambda_{k,j_0}\norm{\psi_k}_{K_{j_0}}<\infty
\end{align*}
which proves the claim. It follows that the limit function $\psi:=\sum_k\psi_k$
is continuous and satisfies $S(\psi)\supset\bigcup_k S(\psi_k)\supset
\mbb{P}TX\setminus C(X)$, hence it is strictly plurisubharmonic on $\Omega$.
\end{proof}

In the following we say that a complex manifold $X$ is \emph{pseudoconvex} if
there is a continuous plurisubharmonic exhaustion function $\rho\colon
X\to\mbb{R}^{>0}$.

\begin{lem}\label{Lem:InnerIntCurve}
Let $X$ be a pseudoconvex complex manifold and let $\gamma\colon U\to X$ be
the integral curve of a holomorphic vector field on $X$ where $U$ is a domain in
$\mbb{C}$. If $\gamma'(U)$ meets $C(X)$, then $\gamma'(U)$ is contained in
$C(X)$. If $X$ admits a smooth plurisubharmonic exhaustion function, then
$\gamma'(U)\subset C(X)$ implies that $\gamma(U)$ is relatively compact in $X$.
In particular, in this case we have $U=\mbb{C}$.
\end{lem}

\begin{proof}
Let $\xi$ be the holomorphic vector field on $X$ with integral curve $\gamma$
and suppose that $\gamma'(0)=\xi(x_0)\in C(X)$. It is enough to show that $0$
is an inner point of the set of $t\in U$ with $\gamma'(t)\in C(X)$,
for then the closed set $(\gamma')^{-1}\bigl(C(X)\bigr)$ is also open, hence
equal to $U$. In other words, we must prove that for every plurisubharmonic
function $\varphi\in\mathcal{C}(X)$ smooth in a neighborhood of $x_t:=\gamma
(t)$ we have $\xi(\varphi)(x_t)=0$ whenever $\abs{t}$ is sufficiently small.

To do this, choose $\alpha\in\mbb{R}^{>0}$ such that $x_0=\gamma(t_0)\in
X_\alpha:=\bigl\{x\in X;\ \rho(x)<\alpha\bigr\}$ where $\rho$ is a continuous
plurisubharmonic exhaustion function of $X$. Let $\Phi^\xi$ be the holomorphic
local flow of $\xi$. For $\abs{t}$ sufficiently small we have
\begin{equation*}
\Phi^\xi_t(X_{\alpha+1})\supset X_\alpha\ni x_t:=\Phi^\xi_t(x_0).
\end{equation*}
Since $\Phi^\xi_t\colon X_\alpha\to X_{\alpha+1}$ is holomorphic,
$\varphi_t:=\varphi\circ\Phi^\xi_t$ is continuous plurisubharmonic on
$X_\alpha$ and smooth in a neighborhood of $x_0$ for each plurisubharmonic
function $\varphi\in\mathcal{C} (X_{\alpha+1})$ that is smooth in a neighborhood
of $x_t$. Following the proof of~\cite[Proposition~1.6]{Hir} we
construct a continuous plurisubharmonic function $\psi_t$ on $X$ which
coincides with $\varphi_t$ in a neighborhood of $x_0$. Choose $\beta\in
\mbb{R}$ such that $\varphi_t(x_0)<\beta<\alpha$ and note that $K:=\rho^{-1}
(\beta)\subset X_\alpha$ is compact. Then choose a convex increasing function
$\chi$ on $\mbb{R}$ fulfilling
\begin{align*}
\chi\bigl(\rho(x_0)\bigr) &< \varphi_t(x_0)\quad\text{ and}\\
\chi(\beta) &> \norm{\varphi_t}_K.
\end{align*}
Finally, define $\psi_t\colon X\to\mbb{R}$ by
\begin{equation*}
\psi_t(x):=
\begin{cases}
\max\bigl(\varphi_t(x), \chi\circ\rho(x)\bigr) &: \rho(x)\leq\beta\\
\chi\circ\rho(x) &: \rho(x)\geq\beta.
\end{cases}
\end{equation*}
One checks directly that $\psi_t$ is continuous plurisubharmonic and
coincides with $\varphi_t$ in some neighborhood of $x_0$. Consequently, we
may calculate
\begin{equation*}
\xi\varphi(x_t)=\left.\frac{d}{ds}\right|_t\varphi\bigl(\Phi^\xi_s
(x_0)\bigr)=\xi\varphi_t(x_0)=\xi\psi_t(x_0)=0
\end{equation*}
since $\xi_{x_0}=\gamma'(0)\in C(X)$. Therefore we see that $\gamma'(t)\in C(X)$
for every $t\in U$ sufficiently close to $0$, which proves the first part of
the lemma.

If $\rho$ is smooth, then choosing $\varphi=\rho$ in the argument given above, we
see that $\gamma(U)$ lies in a fiber of $\rho$, hence is relatively compact.
\end{proof}

\section{Statement of the Main Theorem}

Let us fix $a_1,a_2\in\mbb{C}$ such that $0<\abs{a_1}\leq\abs{a_2}<1$. The
automorphism $\varphi\colon\mbb{C}^2\setminus\{0\}\to\mbb{C}^2\setminus\{0\}$,
$(z_1,z_2)\mapsto(a_1z_1,a_2z_2)$, generates a free proper $\mbb{Z}$-action on
$\mbb{C}^2\setminus\{0\}$. By definition, the compact complex surface
$H_a:=(\mbb{C}^2\setminus\{0\})/\mbb{Z}$ for $a=(a_1,a_2)$ is a \emph{primary
Hopf surface}. We will write $[z_1,z_2]:=\pi(z_1,z_2)$ where
$\pi\colon\mbb{C}^2\setminus\{0\}\to H_a$ is the quotient map.

The torus $T=\mbb{C}^*\times\mbb{C}^*$ acts holomorphically on $H_a$ with three
orbits. More precisely, we have $H_a=E_1\cup H_a^*\cup E_2$ where $H_a^*:=
(\mbb{C}^*\times\mbb{C}^*)/\mbb{Z}$ is the open $T$-orbit, and where
$E_1:=(\mbb{C}^*\times \{0\})/\mbb{Z}=T\cdot[1,0]$ and
$E_2:=(\{0\}\times\mbb{C}^*)/\mbb{Z}=T\cdot[0,1]$ are elliptic curves.

Note that $H_a^*$ is a connected Abelian complex Lie group which thus can be
represented as $\mbb{C}^2/\Gamma_3$ where $\Gamma_3$ is a discrete subgroup of
rank $3$ of $\mbb{C}^2$. The map $p\colon\mbb{C}^2\to\mbb{C}^2/\Gamma_3 \cong
H_a^*$ is the universal covering of $H_a^*$. Let $V$ be the real span of
$\Gamma_3$ and set $W:=V\cap iV$. There are two possibilities. Either $p(W)$ is
dense in $V/\Gamma_3\cong(S^1)^3$, or $p(W)$ is closed, hence compact, hence an
elliptic curve $E$. In the first case, we have $\mathcal{O}(H_a^*)=\mbb{C}$,
i.e., $H_a^*$ is a \emph{Cousin group}, while in the second case
$H_a^*\cong\mbb{C}^*\times E$.

For the following result we refer the reader to~\cite[Chapter~V.18]{BHPV}.

\begin{prop}
The open orbit $H_a^*$ is not Cousin if and only if $a_1^{k_1}=a_2^{k_2}$ for
some relatively prime $k_1,k_2\in \mbb{Z}$.
\end{prop}

\begin{rem}
If there exist relatively prime integers $k_1,k_2$ with $a_1^{k_1}=a_2^{k_2}$,
then we have the elliptic fibration $H_a\to\mbb{P}_1$, $[z_1,z_2]\mapsto
[z_1^{k_1}:z_2^{k_2}]$. The generic fiber is the elliptic curve
$E=\mbb{C}^*/(z\sim cz)$ where $c:=a_1^{k_1}=a_2^{k_2}$. Note that for a generic
choice of $a=(a_1,a_2)$ the open subset $H_a^*$ is a Cousin group.
\end{rem}

Suppose that $H_a^*$ is Cousin and let $\xi\in\lie{t}$ be the generator of the
relatively compact one parameter subgroup $p(W)$. Let $\xi_{H_a}$ be the
holomorphic vector field induced by the $T$-action on $H_a$. One checks directly
that $\xi_{H_a}$ has no zeros in $H_a$, hence defines a holomorphic foliation of
$H_a$. Note that the open subset $H_a^*$ is saturated with respect to
$\mathcal{F}$ and that the leaves of $\mathcal{F}|_{H_a^*}$ are relatively compact
in $H_a^*$. The closure of a leaf $F\subset H_a^*$ in $H_a^*$ is a Levi-flat
compact smooth hypersurface. In fact, these Levi-flat hypersurfaces are the
fibers of the pluriharmonic function $[z_1,z_2]\mapsto
\frac{\log\abs{z_1}}{\log\abs{a_1}}-\frac{\log\abs{z_2}}{\log\abs{a_2}}$ defined
on $H_a^*$, see~\cite{LY}. If $H_a$ is elliptic, then it is foliated by elliptic
curves. Again, $H_a^*$ is saturated with respect to this foliation and the leaves
are compact in $H_a^*$. This shows that in both cases we obtain a (singular)
holomorphic foliation $\mathcal{F}$ of $H_a$ such that the leaves of
$\mathcal{F}|_{H_a^*}$ are relatively compact in $H_a^*$.

We now state the main result of this note.

\begin{thm}\label{Thm:Main}
Let $H_a$ be a primary Hopf surface and let $D\subset H_a$ be a pseudoconvex
domain. If $D$ is not Stein, then $D$ contains with every point $p\in D$ the
topological closure $\ol{F}_p$ of the leaf $F\in\mathcal{F}$ passing through $p$.
\end{thm}

\begin{rem}
For locally pseudoconvex domains having smooth real-ana\-lytic boundary this
result has been obtained by Levenberg and Ya\-maguchi using the theory of
$c$-Robin functions, see~\cite{LY}.
\end{rem}

\section{Existence of plurisubharmonic exhaustions}

In this section we will show that every smoothly bounded locally pseudoconvex
domain $D\subset H_a$ admits a continuous plurisubharmonic exhaustion function.
For this we will modify Hirschowitz' proof of~\cite[Th\'eor\`eme~2.1]{Hir1}.

\begin{prop}\label{Prop:Exhaustion}
Let $D\subset H_a$ be locally pseudoconvex and suppose that neither $E_1$ nor
$E_2$ is a component of $\partial D$.  Then $D$ admits a continuous
plurisubharmonic exhaustion function.
\end{prop}

\begin{rem}
The hypothesis of Proposition~\ref{Prop:Exhaustion} is fulfilled if $D$ is
locally pseudoconvex and smoothly bounded. Hence, Theorem~\ref{Thm:Main} indeed
generalizes the main result of~\cite{LY}.
\end{rem}

\begin{proof}
We define $\Omega:=\bigl\{(x,\xi)\in D\times\lie{t}\; ;\ \exp(\xi)\cdot x\in
D\bigr\}$. By definition, $\Omega$ is an open subset of $D\times\lie{t}$
containing $D\times\{0\}$. Since $D$ is locally pseudoconvex in $H_a$, it
follows that $\Omega$ is locally pseudoconvex in $D\times\lie{t}$.

We define the boundary distance $d\colon D\to\mbb{R}^{>0}$ by
\begin{equation*}
d(x):=\sup\bigl\{r>0;\ \{x\}\times B_r(0)\subset\Omega\bigr\}.
\end{equation*}
It is elementary to check that $d$ is lower semicontinuous. Since $\Omega$ is
locally pseudoconvex in $D\times T$, every point $x\in\ol{D}$ has an open Stein
neighborhood $U$ such that $\Omega\cap(U\times\lie{t})$ is pseudoconvex. Due to
a result of Lelong, see~\cite[Theorem~2.4.2]{Lel}, the function $-\log d$ is
plurisubharmonic on $U\cap D$ and therefore everywhere on $D$.

Note that $-\log d\equiv-\infty$ if and only if $D$ contains $H_a^*$. Since
neither $E_1$ nor $E_2$ is a component of $\partial D$, this implies $D=H_a$ so
that we may exclude this case in the following.

For every $x\in H_a$ the orbit map $\lie{t}\to T\cdot x$, $\xi\mapsto
\exp(\xi)\cdot x$, is open into its image $T\cdot x$. Thus we see that $-\log
d(x)$ goes to infinity as $x$ approaches a point in $\partial D$. Since
$\partial D$ is compact, this implies that $-\log d$ is an exhaustion.

To end this proof, one verifies directly that $d$ is upper semicontinuous in
any point $x\in D$ such that $d(x)\not=\infty$. Therefore, for every constant
$C>0$ the map $\sup(C,-\log d)$ is a continuous plurisubharmonic exhaustion of
$D$.
\end{proof}

\begin{rem}
The proof of Proposition~\ref{Prop:Exhaustion} shows that the polar set given by
$\{-\log d=-\infty\}$ is non-empty if and only of $D$ contains $E_1$ or $E_2$.
\end{rem}

\begin{rem}
In~\cite{DF}, Diederich and Forn\ae ss give an example of a re\-latively compact
pseudoconvex domain in a $\mbb{P}_1$-bundle over a Hopf surface that has smooth
real-analytic boundary but does \emph{not} admit an exhaustion by pseudoconvex
subdomains.
\end{rem}

\section{Proof of Theorem~\ref{Thm:Main}}

Let us start by noting the following simple but important observation.

\begin{lem}
The domain $H_a^*$ admits a smooth plurisubharmonic exhaustion function.
Consequently, if $D\subset H_a$ is pseudoconvex, then $D^*:=D\cap H_a^*$ is
likewise pseudoconvex.
\end{lem}

If $H_a^*$ is Cousin, this lemma follows from~\cite[Pro\-position~2.4]{Hu}. If
$H_a^*$ is not Cousin, then $H_a^*\cong \mbb{C}^*\times E$ clearly has a smooth
plurisubharmonic exhaustion.

Let $D\subset H_a$ be a pseudoconvex domain which is not Stein. If the
pseudoconvex domain $D^*\subset H_a^*$ is not Stein, then $D^*$ is saturated
with respect to the foliation $\mathcal{F}|_{H_a^*}$,
see~\cite[Theorem~3.1]{GMO2}. Since the leaves of $\mathcal{F}$ are relatively
compact orbits of a one parameter subgroup of $T$, continuity of the action map
$T\times H_a\to H_a$ implies that $D$ is saturated with respect to
$\mathcal{F}$ in this case. Therefore, let us assume that $D\not=D^*$ is Stein.
We will complete the proof of Theorem~\ref{Thm:Main} by showing that then $D$ is
Stein as well.

We note first that $D$ cannot contain $E_1$ or $E_2$ if $D^*$ is Stein. Indeed,
due to the continuity of the leaves of $\mathcal{F}$ remarked above, if $E_1$
was contained in $D$, then some of the relatively compact leaves of
$\mathcal{F}$ would lie in $D^*$, contradicting the assumption that $D^*$ is
Stein.

Let us consider the subset $C(D)\subset\mbb{P}TD$ defined in~\eqref{Eqn:C(X)}
where $\pi\colon\mbb{P}TD\to D$ is the projectivized tangent bundle. The proof
of the following lemma relies essentially on the explicit knowledge of the
structure of primary Hopf surfaces.

\begin{lem}\label{Lem:innerintegralcurve}
Let $D\subset H_a$ be a pseudoconvex domain. If $D^*=D\cap H_a^*$ is Stein,
then we have $\pi\bigl(C(D)\bigr)\subset D^*\setminus D$.
\end{lem}

\begin{proof}
Suppose that $\pi\bigl(C(D)\bigr)$ meets $D^*$. Since $H_a^*$ is an Abelian
complex Lie group, it has a biinvariant Haar measure. Therefore, we can apply
the usual convolution technique in order to approximate the continuous
plurisubharmonic exhaustion of $D$ uniformly on compact subsets by smooth ones.
This allows us to apply Lemma~\ref{Lem:InnerIntCurve} to prove existence of a
complex one parameter subgroup $A$ of $T$ and a point $x\in D^*$ such that
$A\cdot x$ is relatively compact in $D$. Note that $A\cdot x$ is not relatively
compact in $D^*$ since the latter is assumed to be Stein.

We claim that that the closure of such a curve $A\cdot x$ in $D$ (and hence $D$
itself) would have to contain $E_1$ or $E_2$, which then, as noted above,
will contradict our assumption that $D^*$ is Stein. In order to prove this
claim, consider $A=\bigl\{( e^{tz_1},e^{tz_2});\ t\in\mbb{C}\bigr\}$ where
$z=(z_1,z_2)\in\mbb{C}^2 \setminus\{0\}$. If $z_1=0$ or $z_2=0$, we see directly
that the closure of $A\cdot x$ contains $E_2$ or $E_1$. Hence, suppose that
$z_1,z_2\not=0$. If $\frac{z_1}{z_2}\notin\mbb{R}$, already the closure of $A$
in $\mbb{C}^2\setminus\{0\}$ contains $\{z_1=0\}\cup\{z_2=0\}$, thus the closure
of $A\cdot x$ in $D$ contains $E_1\cup E_2$ as well in this case. Therefore, we
are left to deal with the case $A=\bigl\{(e^t,e^{\lambda t});\
t\in\mbb{C}\bigr\}$ where $\lambda=\frac{z_2}{z_1}\in\mbb{R}$. Consider the
smooth map $\pi_A\colon\mbb{C}^*\times\mbb{C}^*\to\mbb{R}^{>0}$ defined by
$\pi_A(w_1,w_2) :=\frac{\abs{w_1}^\lambda}{\abs{w_2}}$. The closure $\ol{A}$ of
$A$ in $T=\mbb{C}^*\times\mbb{C}^*$ is contained in the kernel of $\pi_A$. Since
$\ol{A}\cdot x$ is closed and non-compact in $D^*$, we conclude that
$\pi_A(a_1,a_2)=\frac{\abs{a_1}^\lambda}{\abs{a_2}}$ is closed and non-compact
in $\mbb{R}^{>0}$; in particular we have $\abs{a_1}^\lambda\not=\abs{a_2}$. Now
choose $t_m\in\mbb{C}$ such that $a_1^me^{t_m}=c$ for all $m\in\mbb{Z}$. It
follows that
\begin{equation*}
\left|a_2^me^{\lambda t_m}\right|=\left(\frac{\abs{a_2}}{\abs{a_1}^\lambda}
\right)^m\abs{c}^\lambda.
\end{equation*}
Since $\frac{\abs{a_2}}{\abs{a_1}^\lambda}\not=1$, we see that
$\bigl(a_2^me^{\lambda t_m}\bigr)$ converges to $0$ for $m\to\infty$ or
$m\to-\infty$. This proves again that the closure of $A\cdot x$ in $D$ contains
$E_1$ or $E_2$.
\end{proof}

Combining the Lemmas~\ref{Lem:innerintegralcurve} and~\ref{Lem:StrPsh}, we see
that there exists a plurisubharmonic function $\varphi$ on $D$ which is
strictly plurisubharmo\-nic on $D^*$.

\begin{rem}
Due to~\cite{Ri} (see the formulation given in~\cite[Chapter I.5.E]{Dem}), we may
assume without loss of generality that $\varphi$ is smooth on $D^*$.
\end{rem}

As we have noted above, $D$ cannot contain $E_1$ or $E_2$, so that $D\cap E_1$
and $D\cap E_2$ are closed Stein submanifolds of $D$. Therefore, the following
lemma implies that $D$ is Stein, which then completes the proof of
Theorem~\ref{Thm:Main}.

\begin{lem}\label{Lem:Stein}
Let $X$ be a connected complex manifold endowed with a pluri\-subharmonic
exhaustion $\varphi$. Suppose that there exists a closed Stein submanifold $A$ of
$X$ with at most finitely many connected components such that $\varphi$ is
strictly plurisubharmonic on $X\setminus A$. Then $X$ is Stein.
\end{lem}

\begin{proof}
For $\alpha\in\mbb{R}$ write $X_\alpha:=\{\varphi<\alpha\}$. In the first step we
will show that $X_n$ is Stein for all $n\geq1$. Since $A$ is Stein, every
$A_n:=A\cap X_n$ is a closed Stein submanifold of $X_n$.

Due to~\cite{Siu}, we find an open Stein neighborhood $U_{n+1}$ of $A_{n+1}$ in
$X_{n+1}$. Consequently, there exists a strictly plurisubharmonic exhaustion
function $\psi_{n+1}$ on $U_{n+1}$. Let us choose a relatively compact open
neighborhood $V_n$ of $A_n$ in $U_{n+1}$ as well as a cutoff function
$\chi_{n+1}$ which is identically $1$ on $V_n$ and which vanishes near $\partial
U_{n+1}$. Then $\chi_{n+1}\psi_{n+1}\colon X\to\mbb{R}$ is strictly
plurisubharmonic in a neighborhood of $A_n$ and its Levi form is uniformly
bounded from below on $X_n$. Moreover, let $\rho_n$ be a plurisubharmonic
exhaustion of $X_n$ which is strictly plurisubharmonic and smooth on
$X_n\setminus A_n$. Then, for $k$ sufficiently large the function
$\chi_{n+1}\psi_{n+1}|_{X_n}+ke^{\rho_n}$ is a smooth strictly plurisubharmonic
exhaustion function of $X_n$, proving that $X_n$ is Stein.

Since we know now that $X_n$ is Stein, we can apply~\cite[Corollary~1]{Nar2}
which implies that $X_{n-1}$ is Runge in $X_n$ for every $n\geq2$. Therefore $X=
\bigcup_{n\geq1}X_n$ is a Runge exhaustion of $X$ by relatively compact Stein
open subsets, hence $X$ is Stein, see~\cite[Theorem~VII.A.10]{GuRo}.
\end{proof}

\end{document}